\documentclass[11pt]{amsart}
\usepackage{amssymb}

\usepackage[latin1]{inputenc}

\usepackage[all]{xy}

\usepackage{cases}

\usepackage{enumerate} \usepackage{float} 
\newcommand{\Z}{{\mathbb Z}} 

 \newcommand{\F}{{\mathbb F}}

\newcommand{\rk}{\operatorname{rk}\nolimits}

\newcommand{\BG}{\operatorname{BG}\nolimits}

\newcommand{\A}{\ifmmode{\mathcal{A}}\else${\mathcal{A}}$\fi}
\newcommand{\B}{\ifmmode{\mathcal{B}}\else${\mathcal{B}}$\fi}
\newcommand{\C}{\ifmmode{\mathcal{C}}\else${\mathcal{C}}$\fi}
\newcommand{\D}{\ifmmode{\mathcal{D}}\else${\mathcal{D}}$\fi}
\newcommand{\G}{\ifmmode{\mathcal{G}}\else${\mathcal{G}}$\fi}
\newcommand{\I}{\ifmmode{\mathcal{I}}\else${\mathcal{I}}$\fi}
\newcommand{\J}{\ifmmode{\mathcal{J}}\else${\mathcal{J}}$\fi}
\newcommand{\K}{\ifmmode{\mathcal{K}}\else${\mathcal{K}}$\fi}
\renewcommand{\O}{\ifmmode{\mathcal{O}}\else${\mathcal{O}}$\fi}
\renewcommand{\P}{\ifmmode{\mathcal{P}}\else${\mathcal{P}}$\fi}
\newcommand{\U}{\ifmmode{\mathcal{U}}\else${\mathcal{U}}$\fi}
\newcommand{\M}{\ifmmode{\mathcal{M}}\else${\mathcal{M}}$\fi}
\newcommand{\N}{\ifmmode{\mathcal{N}}\else${\mathcal{N}}$\fi}
\newcommand{\Ss}{\ifmmode{\mathcal{S}}\else${\mathcal{S}}$\fi}
\newcommand{\T}{\ifmmode{\mathcal{T}}\else${\mathcal{T}}$\fi}
\newcommand{\Ff}{\ifmmode{\mathcal{F}}\else${\mathcal{F}}$\fi}
\newcommand{\Ll}{\ifmmode{\mathcal{L}}\else${\mathcal{L}}$\fi}
\newtheorem{Thm}{Theorem}[section]
\newtheorem{Conj}{Conjecture}[section]
\newtheorem{Prop}[Thm]{Proposition}

\newtheorem{Lem}[Thm]{Lemma}

\newtheorem*{Maintheorem}{Main Theorem}

\theoremstyle{definition}
\newtheorem{Defi}[Thm]{Definition}
\newtheorem{Rmk}[Thm]{Remark}

\theoremstyle{remark}
\newtheorem*{Not}{Notation}

\usepackage{color}

\theoremstyle{plain}

\title{Cohomology of finite $p$-groups of fixed nilpotency class}

\author{Oihana Garaialde Oca\~{n}a}
\address{Matematika Saila,
Euskal Herriko Unibertsitatearen Zientzia eta Teknologia Fakultatea,
 posta-kutxa 644, 48080 Bilbo, Spain
}
\email{oihana.garayalde@ehu.es}
\author{Jon Gonz\'alez-S\'anchez}
\address{Departamento de Matemáticas, 
Facultad de Ciencia y Tecnología de la Universidad del País Vasco,
 Apdo correos 644, 48080 Bilbao, Spain}
\email{jon.gonzalez@ehu.es}

\date{\today}

\begin{document}

\begin{abstract} Let $p$ be a prime number and let $c,d$ be natural numbers. Then, the number of possible isomorphism types for the mod $p$ cohomology algebra of a $d$-generated finite $p$-group of nilpotency class $c$ is bounded by a function depending only on $p$, $c$ and $d$.

\vspace{2.0mm}

\vspace{2.0mm}
\end{abstract}

\maketitle

\section{Introduction}

The cohomology of a finite group $G$ was initially defined as the cohomology of the associated classifying space $\BG$ \cite{MacLane}. The cohomology functor $H^*(\BG;\Z)$ translates the topological properties of $\BG$ into the algebraic properties of its cohomology group $H^*(\BG;\Z)$, where the latter category is easier to work in. These graded abelian groups can be used to distinguish non homotopy equivalent topological spaces. The above definition of cohomology of groups, however, does not give much information about the algebraic properties of the given group. 

The algebraic interpretations of low dimensional cohomology groups were the first evidence to which kind of algebraic properties such groups encoded. For instance, if $G$ acts trivially on $\Z$, then $H^1(G;\Z)$ is the group of all the group homomorphisms $G\to \Z$ and $H^2(G;\Z)$ classifies central extensions of $G$ by $\Z$ up to equivalence \cite[Chapter IV]{KSBrown82}. Later, Quillen proved the following remarkable result: given a finite group $G$, the Krull dimension of $H^*(G;\F_p)$ is the $p$-rank of $G$ \cite{Quillen}. Here, $\F_p$ denotes the finite field of $p$ elements with trivial $G$-action.

Our study focuses on mod $p$ cohomology rings $H^*(G;\F_p)$ of finite $p$-groups $G$. We search for algebraic properties of finite ($p$-) groups that are `encoded' by their mod $p$ cohomology rings. More precisely, we look for infinitely many isomorphism classes of finite $p$-groups $\{G_i\}_{i\in I}$ that share a common group property and that give rise to only finitely many isomorphism classes of mod $p$ cohomology rings in $\{H^*(G_i;\F_p)\}_{i\in I}$. The following is the main result of the paper. 

\begin{Maintheorem}\label{MainTheorem} Let $p$ be a prime number and let $c$ and $d$ be natural numbers. Then, the number of possible isomorphism types for the mod $p$ cohomology algebra of a $d$-generated $p$-group of nilpotency class $c$ is bounded by a function depending only on $p$, $c$ and $d$.
\end{Maintheorem}

This work has been inspired by the following striking result of J.~F.~Carlson.

\begin{Thm}\cite[Theorem 5.1]{Carlson} Let $m$ be an integer. Then, there are only finitely many isomorphism types of cohomology algebras in $\{H^*(G_i;\F_2)\}$, where $\{G_i\}$ are $2$-groups of coclass $m$.
\end{Thm}

The coclass of a finite $p$-group of size $p^n$ and nilpotency class $c$ is $m=n-c$. In the same paper, J.~F.~Carlson conjectures that the analogous result should hold for the $p$ odd case \cite[Conjecture 6.1]{Carlson} and this conjecture is supported by the classification of finite $p$-groups by their coclass by Leedham-Green \cite{LeedGreen}. There have been several partial proofs of the above conjecture: B. Eick and D. J. Green made the first steps by considering Quillen categories of groups \cite{BettinaDavid2015} and later, the authors of this paper together with Antonio D\'iaz Ramos, made a considerable progress to proving the conjecture \cite{DGG}. However, the conjecture has not been proved in its totality yet.

The dual problem of this conjecture was studied in \cite{DGG1}, where it is proved that the number of isomorphism types of cohomology algebras of $d$-generated finite $p$-groups of nilpotency class smaller than $p$ is bounded by a function that depends only on $p$ and $d$ (see \cite[Theorem 5.1]{DGG1}). The proof of this result is based on three keystones: the Lazard Correspondence, the description of the cohomology ring of powerful $p$-central $p$-groups with the $\Omega$-extension property (see \cite{Weigel}) and some counting arguments in cohomology algebras using spectral sequences (see \cite{Carlson},\cite{DGG}).

This paper is indeed a continuation to \cite{DGG1} and the novelty in the proof of the Main Theorem (compared to \cite[Main Theorem]{DGG1}) is that once the nilpotency class of a finite $p$-group is bigger than $p$, there is no an analogue of the Lazard Correspondence. Instead, we prove that every $d$-generated $p$-group of nilpotency class $c$ contains a powerful $p$-central subgroup with the $\Omega$-extension property (see Proposition \ref{prop:GpOmega} and Lemma \ref{lem:4central}). The cohomology of such $p$-groups is very well-understood (see Theorems \ref{Thm:Weigel} and \ref{Thm:Weigelp2}) and in fact, the  number of their isomorphism types is bounded in terms of $p$ and $d$.  In the last part, using the counting argument result in Section \ref{sec: countingarguments}, we prove the Main Theorem (compare Theorem \ref{Thm: mainresult}).

\begin{Not}\label{Not: reducedandnilpotentcohomology}
Let $G$ be a group, $G^{p^k}$ denotes the subgroup generated by the $p^k$ powers of $G$. The {\it reduced} mod $p$ cohomology ring $H^*(G;\F_p)_{\text{red}}$ is the quotient $H^*(G;\F_p)/nil(H^*(G;\F_p))$, where $nil(H^*(G;\F_p))$ is the ideal of all nilpotent elements in the mod $p$ cohomology. For groups $H,K$, we will write $[H,K]= \langle [x,y]=x^{-1}y^{-1}xy\; |\; x\in H, y\in K\rangle$ to denote the group commutator. 

We say that a function $k=k(a,b,\ldots,A,B,\ldots )$ is $(a,b,\ldots)$-bounded if there exist a function $f$ depending on $(a,b,\ldots )$ such that $k\leq f(a,b\ldots )$. The rank of a $p$-group $G$ is the sectional rank, that is, 
\[
\rk(G)=\text{max} \{d(H)\; |\; \text{for all} \; H\leq G\},
\]
where $d(H)$ is the minimal number of generators of $H$.
\end{Not}

\textbf{Acknowledgements:} The authors were supported by grants, MTM2014-53810-C2-2-P and MTM2017-86802-P, from the Spanish Ministry of Economy and Competitivity and by the Basque Government, grants IT753-13, IT974-16.  Oihana Garaialde Oca\~na was also supported by a postdoctoral grant from the Basque Government, POS$\underline{\ }2017\underline{\ }2\underline{\ }0031$.

\section{Powerful $p$-central $p$-groups}

Following \cite{Weigel}, we define powerful $p$-central groups with the $\Omega$-extension property. We also recall a result about the mod $p$ cohomology algebra of such $p$-groups. Recall that for a $p$-group $G$, we denote $\Omega_1(G)=\langle g\in G\mid g^p=1\rangle$ and  similarly, $\Omega_2(G)=\langle g\in G\mid g^{p^2}=1\rangle$.

\begin{Defi}\label{defi: powerfulpcentralomega} Let $p$ be a prime number and let $G$ be a $p$-group. Then, 
\begin{enumerate}
\item $G$ is \emph{powerful} if $[G,G] \leq G^p$ for $p$ odd and, if $[G,G]\leq G^4$ for $p=2$.
\item $G$ is \emph{$p$-central} if $\Omega_1(G)\leq Z(G)$ for $p$ odd. For $p=2$, we say that $G$ is {\it $4$-central} if $\Omega_2(G)\leq Z(G)$ and moreover, if $\Omega_2(G)=(\Z/4\Z)^r$ for some $r\in \mathbb{N}$, we say that $G$ is {\it $4^*$-central}.
\item $G$ has the \emph{$\Omega$-extension property} ($\Omega$EP for short) if there exists a $p$-central group $H$ such that $G= H/ \Omega_1(H)$.
\end{enumerate}
\end{Defi}

An easy computation shows that if $G$ has the $\Omega$EP, then $G$ is $p$-central. We finish this subsection by describing the 
mod $p$ cohomology algebra of a powerful $p$-central $p$-group with the $\Omega$EP and of a $4^*$-central $2$-group.

\begin{Thm}\label{Thm:Weigel}Let $p$ be an odd prime, let $G$ be a powerful $p$-central $p$-group with the $\Omega$EP and let $d$ denote the $\F_p$-rank of $\Omega_1(G)$. Then,
\begin{itemize}
\item[(a)] $H^*(G;\F_p)\cong \Lambda(y_1, \dots, y_d) \otimes \F_p[x_1, \dots, x_d]$ with $|y_i|=1$ and $|x_i|=2$,
\item[(b)] the reduced restriction map $j_{\text{red}}: H^*(G;\F_p)_{\text{red}} \to H^*(\Omega_1(G);\F_p)_{\text{red}}$ is an isomorphism.
\end{itemize}
\end{Thm}

\begin{proof}
See \cite[Theorem 2.1 and Corollary 4.2]{Weigel}
\end{proof}

\begin{Thm}\label{Thm:Weigelp2} Let $G$ be a $4^*$-central $2$-group and let $d$ denote the $\F_2$-rank of $\Omega_2(G)$. Then, the following are equivalent:
\begin{itemize}
\item[(a)] $H^*(G;\F_p)_{\text{red}} \cong \F_2[x_1, \dots, x_d],$ where $|x_i|=2$ for all $i=1, \dots, d$.
\item[(b)] $G$ has the $\Omega$EP.
\end{itemize}
\end{Thm}

\begin{proof}
See \cite[Theorem 2.1]{Weigel}.
\end{proof}

\section{$p$-groups of fixed nilpotency class}\label{sec: resultspowerful}

The results in this section give a preliminary setting for the proof of the Main Theorem in Section \ref{sec: Mainproof}. We start by proving an easy lemma about $p$-groups of fixed nilpotency class $c$ and rank $r$. We will show that such $p$-groups have a descending normal series of bounded length $h(r,c)$. Then, the proof of the Main theorem will proceed by induction on such a series.

\begin{Lem}\label{lem: normalseries}
Let $p$ be a prime number and let $c,r$ be positive integers. Let $G$ be a finite $p$-group of nilpotency class $c$ and rank $r$. Then, there exists a series of normal subgroups
\[
G=G_0\supseteq G_1 \supseteq \dots \supseteq G_h=\{1\},
\]
such that
\begin{enumerate}
\item[(a)] $G_i\unlhd G$, $G_i/G_{i+1}$ is a cyclic $p$-group,
\item[(b)] if $G/G_{i+1}$ has nilpotency class $l$, then $G_i\leq \gamma_l(G)$ and, in particular, $G_i/G_{i+1}$ is central in $G/G_{i+1}$,
\item[(c)] $h=h(c,r)\leq cr$ is a number that depends only on $c$ and $r$.
\end{enumerate}
\end{Lem}

\begin{proof}
Let $G$ be a finite $p$-group of nilpotency class $c$. That is, the lower central series $\{\gamma_i(G)\}_{i\geq 1}$ satisfies that $\gamma_{c+1}(G)=\{1\}$ and $\gamma_i(G)\neq 0$ for all $i\leq c$. Note that each factor group $A_i:=\gamma_i(G)/\gamma_{i+1}(G)$ is an abelian $p$-group of rank at most $r$ and it is central in $G/\gamma_{i+1}(G)$. Moreover, each $A_{i}$ has a decomposition with cyclic factor groups, say,
\[
A_{i}=B_{0}^i\supseteq B_{1}^i\supseteq \dots \supseteq B_{l_i}^i=\{1\},
\]
where $B_{j}^i/B_{j+1}^i$ is cyclic and the values $l_i$ depend on the rank of $A_i$ (which in turn depends on the rank of $G$). Write $B_i^j:= C_i^j/\gamma_{i+1}(G)$ for each $i\in \{1, \dots, c\}$ and $j\in \{1, \dots, l_i\}$. Then, we can form a normal series
\[
G=\gamma_1(G)\supseteq C_{1}^1\supseteq \dots \supseteq C_{1}^{l_1}=\gamma_2(G)\supseteq C_{2}^1 \supseteq \dots \supseteq \gamma_{c+1}(G)=\{1\},
\]
where the length of the series depends on $c$ and $r$. For simplicity, denote the above series by $\{G_i\}_{i=0}^{h}$. Finally, it is clear by construction that if  $G/G_{i+1}$ has nilpotency class $l$, then $G_{i+1}\leq \gamma_l(G)$ and in particular, $G_i/G_{i+1}$ is central in $G/G_{i+1}$. Moreover, since the rank of $G$ is $r$, each factor $\gamma_i(G)/\gamma_{i+1}(G)$ has rank at most r and the bound on $h(c,r)$ is attained. 
\end{proof}

We continue by proving a result in commutator subgroups of finitely generated pro-$p$ groups. As usual in the category of (finitely generated) pro-$p$ groups, generation is considered topologically and subgroups are closed unless otherwise stated. 

\begin{Lem}\label{lemma: commutatorlemma}
Let $c$ and $k$ be non-negative integers. Let $G$ be a finitely generated pro-$p$ group of nilpotency class $c$ and let $K=G^{p^k}\leq G$.  
\begin{itemize}
\item[(a)] Suppose that $p$ is odd and that $k-1\geq \log_p({c+1})$, then 
$$[N,K]\leq N^p,$$
for all normal subgroups $N$ of $G$ and in particular, $K$ is powerful.
\item[(b)] Suppose that $p=2$ and that $k-2\geq \log_p({c+1})$, then 
$$[N,K]\leq N^4,$$
for all normal subgroups $N$ of $G$ and in particular, $K$ is powerful.
\end{itemize}
\end{Lem}

\begin{proof}
Note that the group $K$ is a finite index subgroup of $G$, therefore it is open subgroup and also a closed subgroup of $G$.
We start with the odd case. By Theorem 2.4 in \cite{FAGSJZ}, we have that for all normal subgroups $N$ of $G$,
\begin{align*}
[N,K]=[N,G^{p^k}]&\equiv [N,G]^{p^k}(\text{mod} [N,{}_pG]^{p^{k-1}}[N,{}_{p^2}G]^{p^{k-2}}\dots [N,{}_{p^{k}}G])\\
&\leq N^{p^k} (\text{mod} [N,{}_pG]^{p^{k-1}}[N,{}_{p^2}G]^{p^{k-2}}\dots [N,{}_{p^{k}}G]),
\end{align*}and since $[N,{}_{p^i}G]^{p^{k-i}}\leq N^{p^{k-i}}$ for $i=0, \dots, k-1$ and $[N,{}_{p^k}G]\leq \gamma_{c+1}(G)=1$, the statement follows.

If $p=2$, again by Theorem 2.4 in \cite{FAGSJZ}, we have that
\begin{align*}
[N,K]=[N,G^{2^k}]&\equiv [N,G]^{2^k}(\text{mod} [N,{}_2G]^{2^{k-1}}[N,{}_{2^2}G]^{2^{k-2}}\dots [N,{}_{2^{k}}G])\\
&\leq N^{2^k} (\text{mod} [N,{}_2G]^{2^{k-1}}[N,{}_{2^2}G]^{2^{k-2}}\dots [N,{}_{2^{k}}G]).
\end{align*}Now, since $[N,{}_{2^i}G]^{2^{k-i}}\leq N^{2^{k-i}}$ for $i=0, \dots, k-2$ and both $[N,{}_{2^k}G]$ and $[N,{}_{2^{k-1}}G]^2$ are contained in $\gamma_{c+1}(G)=1$, the statement follows.
\end{proof}

\begin{Not}\label{Not: p*}
For a prime number $p$, set $p^*=p$ if $p$ is odd and $p^*=4$ if $p=2$.
\end{Not}

Using the above result, we prove that a finite $p$-group $G$ of fixed nilpotency class $c$ contains a powerful $p^*$-central $p$-subgroup $B$ with the $\Omega$EP. The existence of this subgroup implies that the restriction map from $H^*(G;\F_p)$ to $H^*(B;\F_p)$ is non-trivial (see Theorems \ref{Thm:Weigel} and \ref{Thm:Weigelp2}) and this is crucial to apply Theorem \ref{thm:cohringsfromquotient} in the proof the of the Main result. 

\begin{Prop}
\label{prop:GpOmega}
Let $c$ and $k$ be positive integers and let $p$ be a prime number. Let $G$ be a finite $p$-group of nilpotency class $c$, let $C=\langle y\rangle$ be a cyclic subgroup of $G$ such that $C\leq \gamma_c(G)$. Set $K:=G^{p^k}$.
\begin{itemize}
\item[(a)] If $p$ is odd, suppose that $k-1\geq \log_p({c+1})$, then there exists a $p$-group $H$ such that 
\begin{enumerate} 
\item[(a1)] $H$ is powerful $p$-central.
\item[(a2)] $K\cdot C\cong H/\Omega_1(H)$, that is, $K\cdot C$ is a powerful $p$-central $p$-group with the $\Omega$EP.
\end{enumerate}
\item[(b)] If $p=2$, suppose that $k-2\geq \log_2({c+1})$, then there exists a $2$-group $H$ such that
\begin{enumerate}
\item[(b1)] $H$ is powerful $4$-central.
\item[(b2)] $K\cdot C\cong H/\Omega_1(H)$, that is, $K\cdot C$ is a powerful $4$-central $2$-group with the $\Omega$EP.
\end{enumerate}
\end{itemize}
\end{Prop}

\begin{proof}
Let $X$ be a minimal system of generators of $G$ and let $F$ denote the free pro-$p$ group on $X$ of nilpotency class $c$. Let $\pi :F\to G$ denote the natural projection map, let $N$ denote its kernel and let $x$ be a preimage of $y$ contained in $\gamma_c(F)$. From now on, the proof of the proposition deals with the $p$ odd case and the $p=2$ case, separately.

(a) Suppose that $p$ is an odd prime and let $k$ be an integer such that $k-1\geq \log_p({c+1})$. Put $M:= N\cap \langle x,F^{p^k}\rangle$ and let $H:=\langle x,F^{p^k}\rangle /M^p$ be a finite $p$-group. Then, 
$$
K\cdot C=G^{p^k}\cdot \langle y\rangle\cong \left(\frac{F}{N}\right)^{p^k} \cdot \frac{\langle x,N\rangle}{N}=\frac{\langle x,F^{p^k}\rangle N}{N}\cong \frac{\langle x,F^{p^k}\rangle}{N\cap \langle x,F^{p^k}\rangle }=\frac{\langle x, F^{p^k}\rangle}{M}.
$$
Taking into account that $x$ is contained in the center of $F$ and applying the previous lemma to $N=F^{p^k},$ we obtain that
$$
[\langle x,F^{p^k}\rangle,\langle x,F^{p^k}\rangle]= [F^{p^k},F^{p^k}]\leq (F^{p^k})^p.
$$ 
Then, $\langle x,F^{p^k}\rangle$ is a powerful pro-$p$ group and whence, $K\cdot C=G^{p^k}\cdot C$ is also a powerful $p$-group. Moreover, $M$ is a normal subgroup of $F$ that is contained in $\langle x,F^{p^k}\rangle$ and
$$[M,M]\leq [M,\langle x,F ^{p^k}\rangle]\leq[M,F ^{p^k}]\leq M^p,$$
where in the first inequality we used that $x$ is central and in second inequality we used the previous lemma with $N=M$. This inequality yields that $M$ is a finitely generated powerful pro-$p$ group. Moreover, as $F$ and $M$ are torsion-free (Compare \cite[Theorem 5.6 and Corollary bellow]{Phall} and \cite[Theorem 10.14]{Warf}), we have that $M$ is a uniform pro-$p$ group \cite[Theorem 4.5]{DDMS}. Thus, $M^p=\{ m^p\mid m\in M\}$ (see \cite[Theorem 3.6]{DDMS}). In particular, by \cite[Lemma 4.10]{DDMS}, 
$$\Omega_1(H)=\Omega_1(\langle x,F^{p^k}\rangle/M^p)=M/M^p,$$ 
and 
$$[\Omega_1(H),H]=[M/M^p,\langle x,F^{p^k}\rangle/M^p]=[M,F^{p^k}]M^p/M^p=1.$$ 
Therefore, $H$ is a powerful $p$-central $p$-group  and 
$$H/\Omega_1(H)=\frac{\langle x,F^{p^k}\rangle /M^p}{M/M^p}\cong \frac{\langle x,F^{p^k}\rangle}{M}\cong \langle y,G^{p^k}\rangle=K\cdot C.$$ 
The last equality shows that $K\cdot C$ has the $\Omega$EP and by \cite[Corollary 2.4]{GW}, we conclude that $G^{p^k}\cdot C$ is a powerful $p$-central $p$-group with the $\Omega$EP. 

(b) Suppose that $p=2$ and let $k$ be an integer such that $k-2 \geq \log_2({c+1})$. Construct $M$ and $H$ as in (a) by taking $p=2$. Also, take $N=F^{2^k}$ in the previous lemma to obtain that $[\langle x,F^{2^k}\rangle,\langle x,F^{2^k}\rangle]=[F^{2^k},F^{2^k}]\leq (F^{2^k})^4$. Then, $\langle x,F^{2^k}\rangle$ is a powerful pro-$2$ group and whence, $K\cdot C=G^{2^k}\cdot C$ is also a powerful $2$-group. 

Note that again $M$ is a normal subgroup of $F$ that is contained in $\langle x, F^{2^k}\rangle$. Then,
$$[M,M]\leq [M, \langle x, F^{2^k}\rangle]\leq [M,F ^{2^k}]\leq M^4,$$
where in the first inequality we used the fact that $x$ is central and in second inequality we used the previous lemma with $N=M$. This in turn shows that $M$ is a finitely generated powerful pro-2 group. As before, $F$ and $M$ are torsion-free and we have that $M$ is a uniform pro-$2$ group. Thus, $M^4=\{ m^4\mid m\in M\}$ (see \cite[Theorems 4.5, 3.6]{DDMS}). In particular, by \cite[Lemma 4.10]{DDMS}, 
$$\Omega_2(H)=\Omega_2(\langle x,F^{2^k}\rangle/M^4)=M/M^4,$$ 
and 
$$[\Omega_2(H),H]=[M/M^4,\langle x,F^{2^k}\rangle/M^4]=[M,F^{2^k}]M^4/M^4=1.$$ 
Therefore, $H$ is a powerful $4$-central $2$-group. Notice that by  \cite[Corollary 2.4]{GW} the group $H/\Omega_1(H)$ is also a powerful $4$-central $2$-group. 
Furthermore 
$$\frac{H/\Omega_1(H)}{\Omega_1(H/\Omega_1(H))}\cong H/\Omega_2(H)=\frac{\langle x,F^{2^k}\rangle /M^4}{M/M^4}\cong \frac{\langle x, F^{2^k}\rangle}{M}\cong \langle y,G^{2^k}\rangle=K.C.$$ 
The last equality shows that $K$ has the $\Omega$EP and by  \cite[Corollary 2.4]{GW}, we conclude that $G^{2^k}$ is a powerful $4$-central $2$-group with the $\Omega$EP. The last property easily follows by construction.
\end{proof}

\begin{Rmk}
In the above result, if $G$ has a bounded number of minimal generators $d$, then the index $|G:K\cdot C|$ is $(p,d,c)$-bounded; a crucial condition to apply Theorem \ref{thm:cohringsfromquotient} to  $K\cdot C$.
\end{Rmk}

For $p=2$, we would also like to conclude that the aforementioned restriction map $H^*(G;\F_p)\to H^*(B;\F_p)$ is non-trivial. To that aim, $B$ should additionally be a $4^*$-central subgroup (see Theorem \ref{Thm:Weigelp2}).

\begin{Lem}
\label{lem:4central}
Let $G$ be a $d$-generated $4$-central powerful $2$-group with the $\Omega$EP. Then, there exists a normal subgroup $N$ of $G$ such that $|G:N|\leq 2^d$ and $N$ is a $4^*$-central powerful $2$-group with the $\Omega$EP. Furthermore, if $C$ is a cyclic central subgroup of $G$ of size bigger than $2$, then $N\cdot C$ is also a $4^*$-central subgroup with the $\Omega$EP.
\end{Lem}

\begin{proof}
Let $\pi :G/G^2\to G^2/G^4$ be the homomorphism defined by $\pi (xG^2)=x^2G^4$. Let $V$ be the kernel of $\pi$ and let $W$ be a complement of $V$ in $G/G^2$. Since $\pi$ is surjective the restriction $\pi_{W} :W\to G^2/G^4$ is an isomorphism. Let $x_1G^2,\ldots ,x_lG^2$ be generators of $W$ and put $N=\langle x_1,\ldots ,x_l\rangle$. By construction, $N^2=G^2$ and in particular, $N$ is a normal subgroup of $G$. Furthermore, $[N,N]\leq [G,G]\leq G^4=N^4$ and thus, $N$ is a powerful subgroup of $G$.

Let us check that $N$ is $4^*$-central. Since $N$ is powerful $4$-central, it suffices to show that $\Omega_1(N)=\Omega_1(N^2)$. Observe that $\Omega_1(N^2)\subseteq \Omega_1(N)$. Now, if $g\in N$ and $g^2=1$, then $gG^2\in V$. As $W$ is a complement of the kernel $V$ of $\pi$, this can only happen if $gG^2=1.G^2$. That is, $g$ must be contained in $G^2=N^2$. Thus, $N$ is a powerful $4^*$-central subgroup and so is $N\cdot C$. Finally, note that the $\Omega$EP property is closed under taking subgroups and this finishes the proof.
\end{proof}

\section{A counting argument using spectral sequences}\label{sec: countingarguments}

In this section, we shall consider a central extension $C\to G\to Q$ of finite $p$-groups with $C$ a cyclic $p$-group and show that under mild assumptions, the cohomology algebra $H^*(G;\F_p)$ is determined up to a finite number of possibilities by the cohomology algebra $H^*(Q;\F_p)$. Henceforth, by a spectral sequence associated to a central extension, we mean the Lyndon-Hochschild-Serre spectral sequence (LHS for short) \cite[VII.6]{KSBrown82}. The following result can be considered as a generalization of Theorem 4.2 in \cite{DGG1} and it is crucial in the proof of the main theorem. 

\begin{Thm}\label{thm:cohringsfromquotient}
Let $p$ be a prime number, let $k, c, r, f$ be positive integers and suppose that 
\begin{equation}
\label{eq:cohringsfromquotient}
1 \to C \to G \to Q \to 1
\end{equation}
is a central extension of finite $p$-groups where $C$ a cyclic group of order $p^k$ with $k\geq 1$, $\rk(G)\leq r$ and $Q$ has a subgroup $A$ of nilpotency class $c$ with $|Q:A| \leq f$. Then the ring $H^*(G;\F_p)$ is determined up to a finite number of possibilities (depending on $p$, $k$, $c$, $r$ and $f$) by the ring $H^*(Q;\F_p)$.
\end{Thm}

\begin{proof}
Let $C=C_{p^k}$ denote a cyclic $p$-group of size $p^k$. We start the proof by treating separately the $p=2$ and $k=1$ case. 

Suppose that $k\geq 1$ if $p$ is odd or $k\geq 2$ if $p=2$. By Proposition \ref{prop:GpOmega} and Lemma \ref{lem:4central}, 
there exists a powerful $p$-central or $4^*$-central subgroup $B$ of $G$ with the $\Omega$EP for $p$ odd or $p=2$, respectively, that contains $C$. The index of $B$ in $G$ is bounded in terms of $p$, $r$ and $f$.  By Theorems \ref{Thm:Weigel}(b) and \ref{Thm:Weigelp2}, there exists a class $\eta\in H^2(B;\F_p)$ such that $\text{res}^B_C( \eta )$ is non-zero.

Suppose that $p=2$ and $k=1$, then $C=C_2$ is either contained in $N$ constructed in Lemma \ref{lem:4central} or $C$ is not contained in $N$. In the former case,   considering $B=N$, we can proceed as in the previous paragraph and there exists a class $\eta\in H^2(B;\F_p)$ such that $\text{res}^B_C( \eta )$ is non-zero. In the latter case, $C\cap N=\{1\}$ and put $B= N\cdot C\cong N\times C_2$. In this case there also exists a class $\eta\in H^2(B;\F_p)$ such that $\text{res}^B_C( \eta )$ is non-zero.

Now, following the arguments of the proof of \cite[Lemma 3.2]{Carlson}, we shall show that the spectral sequence $E$ arising from 
\begin{equation}\label{cyclicextensionfromquotient}
1 \to C \to G \to Q \to 1,
\end{equation}
stops at most at page $2{|G:B|}+1$. As the extension \eqref{cyclicextensionfromquotient} is central, the second page $E_2$ of the spectral sequence is isomorphic to the tensor product (as bigraded algebras),
\[
E_2^{*,*}\cong H^*(Q;\F_p)\otimes H^*(C;\F_p).
\]
Let $\zeta= \N orm^G_B(\eta) \in H^{2|G:B|}(G;\F_p),$ where $\N orm(\cdot)$ is the Evens norm map. Write $|G:B|=p^n$ for short. From the usual Mackey formula for the double coset decomposition $G= \underset{x \in D}\cup C x B$ with $D=C\setminus G/B$, 
\begin{align*}
\text{res}^G_C(\N orm^G_B(\eta))=\text{res}^G_C(\zeta)= \prod_{x \in D} \N orm^{C \cap xBx^{-1}}_C(\text{res}^{xBx^{-1}}_{C \cap xBx^{-1}}(\eta)),
\end{align*}
we have that $\text{res}^G_C(\zeta)\neq 0$. The restriction map on cohomology from $G$ to $C$ is the edge homomorphism on the spectral sequence and thus, the image of the restriction map 
$$
\text{res}^G_C: H^{2p^n}(G;\F_p) \to H^{2p^n}(C;\F_p)
$$
is isomorphic to $E_{\infty}^{0,2p^n}$. Let $\zeta' \in E_{\infty}^{0, 2p^n}$ be an element representing $\text{res}_C^G(\zeta)$. Suppose that $t \geq 2p^n+1$ and let $\mu \in E_t^{r,s}$ with $s=a(2p^n)+b$ where $b < 2p^n$. We may write $\mu= (\zeta')^a\mu'$ for some $\mu' \in E_t^{r,b}$. Then,
\[
d_t(\mu)= d_t((\zeta')^a \mu')=d_t((\zeta')^a)\mu'+(\zeta')^a d_t(\mu')=(\zeta')^ad_t(\mu')=0
\]
as $d_t(\mu') \in E_t^{r+t, b+1-2p^n}$ with $b+1-2p^n<0$. Then, $\zeta'$ is a regular element on $E_t^{*,*}$, in turn, $\zeta$ is regular on $H^*(G;\F_p)$ and $d_t=0$ for all $t \geq 2p^n+1$. Thus, the spectral sequence collapses at most at page $2p^n+1$. We shall finish the proof by following the proof of \cite[Proposition 3.1]{Carlson}.

Let $\gamma_1, \dots, \gamma_u$ be elements in $H^*(C;\F_p)$ generating a polynomial subalgebra over which $H^*(C;\F_p)$ is a finitely generated module. Note that if an element $\gamma \otimes 1 \in H^s(C;\F_p)\otimes 1  \subset E_n^{0,*}$ survives to the $n^{\text{th}}$ page then, as $d_n$ is a derivation,
\[
d_n(\gamma^p)=p\gamma^{p-1}d_n(\gamma)=0.
\]
Since the spectral sequence stops at page $2p^n +1$, we must have that for all $i$, $\tau_i=\gamma_i^{p^{2p^n+1}}$ is a universal cycle. Analogously, let $\tau_{u+1}, \dots, \tau_{v}$ be homogeneous parameters for $H^*(Q;\F_p) \otimes 1 \subset E_n^{*,0}$. Then, $E_2$ is a finitely generated module over the polynomial subalgebra $W$ generated by $\tau_1, \dots, \tau_v$. Moreover, $d_j$ is a $W$-module homomorphism because $d_j(\tau_i)=0$ and thus, $E_{j+1}^{*,*}$ is finitely generated $W$-module. As Carlson says, such generators $\alpha_1, \dots, \alpha_q$ of $E_j^{*,*}$ can be chosen to be homogeneous and the key point is that then $d_j$ is determined by $d_j(\alpha_1), \dots, d_j(\alpha_v)$. For each $i$, there is only a finite number of choices for the images $d_j(\alpha_i)$. So, for all $j \geq 2$, there are finitely many choices for $d_j$ and for the $W$-modules and algebra structures of $E_{j+1}$. Since the spectral sequence stops after finitely many steps, there are finitely many $W$-module and algebra structures for $E_{\infty}^{*,*}$. Now, the result holds from \cite[Theorem 2.1]{Carlson}. 

\end{proof}

\section{Proof of the Main result and further work}\label{sec: Mainproof}

The aim of this section is to prove the Main theorem of this paper using the results in Sections \ref{sec: resultspowerful} and \ref{sec: countingarguments}. 

\begin{Thm}\label{Thm: mainresult}
 Let $p$ be a prime number and let $c$ and $d$ be natural numbers. Then, the number of possible isomorphism types for the mod $p$ cohomology algebra of a $d$-generated $p$-group of nilpotency class $c$ is bounded by a function depending only on $p$, $c$ and $d$.
\end{Thm}

\begin{proof}Let $G$ be a $d$-generated $p$-group of nilpotency class $c$. Let $r$ denote the sectional rank of $G$ that depends on $d$ and $c$ (see Notation in Section \ref{Not: reducedandnilpotentcohomology}). Then, by Lemma \ref{lem: normalseries}, there is a central series,
\[
G=G_0\supseteq G_1\supseteq \dots \supseteq G_h=\{1\},
\]
where $h=h(c,r)\leq cr$ is a finite number (depending only on $c$ and $r$) and the quotients $G_i/G_{i+1}$ are cyclic. Then, for each $i=1,\dots, h-1$, consider the following central extension
\begin{equation}\label{eq: centralextensionofnormalseries}
G_i/G_{i+1}\to G/G_{i+1}\to G/G_i,
\end{equation}
where $G_i/G_{i+1}\cong C_{p^{k_i}}$ is a cyclic group. Note that the rank of $G/G_{i+1}$ is at most  $r$ and that $G/G_{i}$ is a $p$-group of nilpotency class smaller than $c$. 
Then, the theorem easily follows by reversed induction on $i$ and Theorem \ref{thm:cohringsfromquotient}.
\end{proof}

We finish this paper with some comments and further questions. Finite $p$-groups of fixed nilpotency class and finite $p$-groups of fixed coclass share the following property: they all have bounded rank and therefore, they contain a powerful $p$-central subgroup of bounded index. Using  the correspondence with Lie algebras in \cite{WeigelHS}, one can easily show that these $p$-groups contain a powerful $p$-central $p$-subgroup with the $\Omega$EP for $p\geq 5$. It is natural to expect that this also holds for the primes $p=2$ and $p=3$.

\begin{Conj} Let $G$ be a finite $p$-group of fixed rank $r$. Then, there exists a subgroup $K$ of bounded index by a function depending on $p$ and $r$ such that it is powerful $p^*$-central with the $\Omega$EP.
\end{Conj}

For $p$ odd, the group $K$ above has a particular type of mod $p$ cohomology algebra and powerful $4^*$-central $2$-groups with the $\Omega$EP should have the same cohomological structure.

\begin{Conj} Let $G$ be a powerful $4^*$-central $2$-group with the $\Omega$EP property and let $d=\rk (\Omega_1(G))$ Then,
\[
H^*(G;\F_2)\cong \Lambda (y_1,\ldots ,y_d)\otimes \F_2[x_1, \dots, x_d].
\]
where the generators $y_i$ have degree one  and the generators $x_i$ have degree two.
\end{Conj}

In \cite{DGG1}  the authors conjectured that for a given integer $r,$ there only finitely many possibilities for the mod $p$ cohomology algebra of a finite $p$-group of rank $r$.
We expect that the previous two conjectures and proving that the Lydon-Hochschild-Serre spectral sequence arising from the extension of $p$-groups of rank at most $r$,
\[
1\to N\to G\to C_p\to 1
\]
collapses in a bounded number of steps, would prove Conjecture 5.2 in \cite{DGG1}.

\end{document}